\documentclass[a4paper,11pt]{amsart}
\usepackage{mathptmx}
\usepackage[all]{xy}
\usepackage[colorlinks=true]{hyperref}
\usepackage{amsmath}
\usepackage{enumerate}
\usepackage{amssymb}
\usepackage{combelow}
\usepackage{tensor}
\usepackage{mathrsfs}

\usepackage{soul}
\newtheorem{thm}{Theorem}[section]

\newtheorem{prop}[thm]{Proposition}
\newtheorem{lem}[thm]{Lemma}
\newtheorem{cor}[thm]{Corollary}

\theoremstyle{definition}
\newtheorem{defn}[thm]{Definition}
\newtheorem{ex}[thm]{Example}

\theoremstyle{remark}
\newtheorem{rem}[thm]{Remark}

\newcommand{\RR}{\mathbb R}
\newcommand{\Ss}{\mathbb S}
\newcommand{\Zz}{\mathbb Z}

\newcommand{\Tt}{\mathbb T}
\renewcommand{\d}{\mathrm d}

\newcommand{\eps}{\varepsilon}

\newcommand{\X}{\ensuremath{\mathfrak{X}}}
\newcommand{\F}{\ensuremath{\mathcal{F}}}

\renewcommand{\d}{\mathrm d}               
\newcommand{\Lie}{\mathscr{L}}    

\newcommand{\smalcirc}{\mbox{\,\tiny{$\circ $}\,}}  

\DeclareMathOperator{\rank}{rank}       
\DeclareMathOperator{\codim}{codim}     
\DeclareMathOperator{\Ker}{Ker}         
\DeclareMathOperator{\im}{Im}           
\DeclareMathOperator{\pr}{pr}     

\newcommand{\al}{\alpha}
\newcommand{\be}{\beta}

\newcommand{\G}{\mathcal{G}}            
\newcommand{\K}{\mathcal{K}}            

\newcommand{\s}{\mathbf{s}}             
\renewcommand{\t}{\mathbf{t}}           
\newcommand{\m}{\mathbf{m}}             
\renewcommand{\gg}{\mathfrak{g}}        
\newcommand{\kk}{\mathfrak{k}}          
\newcommand{\ssuu}{\mathfrak{s}\mathfrak{u}}

\newcommand{\tto}{\rightrightarrows}    
\newcommand{\ract}{\curvearrowright}

\DeclareMathOperator{\Ad}{Ad}           

\newcommand*{\der}[1]{\overrightarrow{#1}}
\newcommand*{\esq}[1]{\overleftarrow{#1}}
\newcommand{\diffto}{\xrightarrow{\raisebox{-0.2 em}[0pt][0pt]{\smash{\ensuremath{\sim}}}}}

\newcommand{\timesst}{\tensor[_\s]{\times}{_\t}} 

\usepackage[dvips]{graphicx}

\begin{document}

\title{Cosymplectic groupoids}

\author{Rui Loja Fernandes}
\address{Department of Mathematics, University of Illinois at Urbana-Champaign, 1409 W. Green Street, Urbana, IL 61801, USA}
\email{ruiloja@illinois.edu}

\author{David Iglesias Ponte}
\address{ULL-CSIC Geometr\'{\i}a Diferencial y Mec\'anica Geom\'etrica, Departamento de Mate\-m\'aticas, Estad\'{\i}stica e Investigaci\'on Operativa and Instituto de Matem\'aticas y Aplicaciones (IMAULL), University of La Laguna, San Crist\'obal de La Laguna, Spain}
\email{diglesia@ull.edu.es}



\thanks{RLF was partially supported by NSF grant DMS-2003223.} 

\begin{abstract}
A cosymplectic groupoid is a Lie groupoid with a multiplicative co\-sym\-plectic structure. We provide several structural results for cosymplectic grou\-poids and we discuss the relationship between cosymplectic groupoids, Poisson groupoids of corank 1, and oversymplectic groupoids of corank 1.
\end{abstract}

\maketitle

\section{Introduction}             %
\label{sec:introduction}           %

A \emph{cosymplectic groupoid} is a Lie groupoid $\G\tto M$ equipped with a multiplicative cosymplectic structure $(\omega,\al)$. This means that $\omega\in\Omega^2(\G)$, $\alpha\in\Omega^1(\G)$ are closed multiplicative forms and $\al\wedge\omega^m$ is nowhere vanishing, so it defines a volume form in $\G$. As we will see later, one must have  $\dim\G=2m+1$ where $m=\dim M$ (we assume $M$ connected).  The notion of cosymplectic groupoid was first studied in \cite{DjWa}.

Cosymplectic groupoids lie at the intersection of two well-known, interesting, classes of Lie groupoids:
\begin{itemize}
\item \emph{Poisson groupoids}, i.e., Lie groupoids with a multiplicative Poisson structure (see, e.g., \cite{MX94,We}). For a cosymplectic groupoid $(\G,\omega,\alpha)$ the associated multiplicative Poisson structure $\pi_\G\in\X^2(\G)$ has symplectic foliation $\ker\al$ and leafwise symplectic form the restriction of $\omega$.
\item \emph{Oversymplectic groupoids}, i.e., Lie groupoids with a closed multiplicative 2-form $\omega$ satisfying $\rank\omega_{1_x}=2\dim M$, for all $x\in M$ (see, e.g., \cite{BCWZ}). We will see that the 2-form of a cosymplectic groupoid satisfies this condition.
\end{itemize}
For both of these classes of Lie groupoids the base $M$ inherits a Poisson structure. For a cosymplectic groupoid $(\G,\omega,\alpha)$ the Poisson structures obtained from $\pi_\G$ and from $\omega$ coincide, and will be denoted by $\pi_M\in\X^2(M)$. Our aim in this paper is to give structural results for cosymplectic groupoids and to establish precise relationships with these two classes of Lie groupoids.

To describe our main results we observe that, as a consequence of the multiplicativity condition, the standard data associated with a cosymplectic structure satisfies:
\begin{enumerate}[(i)]
\item $\ker\al\subset T\G$ defines an integrable \emph{multiplicative} distribution in $\G$;
\item $\omega$ restricts to a symplectic form on the leaves of $\ker\al$, yielding a  \emph{multiplicative} Poisson structure $\pi_\G\in\X^2(\G)$;
\item The Reeb vector field $E\in \X(G)$, characterized by $i_E\omega=0$ and $i_E\al=1$, is \emph{bi-invariant} (i.e., it is both left and right invariant). In particular, it is a complete vector field.
\end{enumerate}
These basic facts, to be proved below, give a rich structure to a cosymplectic groupoid. For example, the collection of all orbits of the Reeb vector field intersecting the identity section is a bundle of Lie groups $\K\subset \G$ and one has:

\begin{thm}
\label{thm:main:grpd}
For any cosymplectic groupoid $(\G,\omega,\al)$ there is a short exact sequence of topological groupoids over the same base:
\begin{equation}
\label{eq:seq:cosymp:grpd}
\xymatrix{1\ar[r]& \K\ar[r] & \G\ar[r] & \Sigma \ar[r] &1}
\end{equation}
where $\Sigma$ is the orbit space of the Reeb vector field. When this space is smooth, this is a short exact sequence of Lie groupoids and $(\Sigma,\Omega)$ is a symplectic groupoid for a unique symplectic structure making the projection $\G\to\Sigma$ a Poisson map.
\end{thm}

{
For an arbitrary Poisson groupoid or oversymplectic groupoid, the base Poisson structure $(M,\pi_M)$ may fail to be integrable. However, for a cosymplectic groupoid it is always integrable. Indeed, the identity section of a cosymplectic groupoid $(\G,\omega,\alpha)$ is contained in a single symplectic leaf of the Poisson structure $\pi_\G$ and we have:

\begin{prop}
\label{prop:identity:leaf}
Let  $(\G,\omega,\al)$  be a cosymplectic groupoid. The symplectic  leaf of $\pi_\G$ containing $M$ is a Lie subgroupoid $\Sigma^0\subset \G$, and it yields a symplectic groupoid  $(\Sigma^0,\omega|_{\Sigma^0})\tto M$ integrating $(M,\pi_M)$.
\end{prop}

For a proper cosymplectic groupoid where $\Sigma^0$ is an embedded submanifold, one obtains a picture somewhat dual to Theorem \ref{thm:main:grpd}. Namely, the flow of the Reeb vector field for some fixed time $t_0$ gives a symplectomorphism of $\Sigma^0$ and one finds that the cosymplectic groupoid is a symplectic mapping torus.

\begin{thm}
\label{thm:main:grpd:proper}
Let $(\G,\omega,\al)$ be a proper, source connected, cosymplectic groupoid and assume that the symplectic leaf $\Sigma^0\subset \G$ containing the identity is an embedded submanifold. Then there is a symplectomorphism $\varphi:\Sigma^0\to\Sigma^0$ such that $\G$ is isomorphic to the symplectic mapping torus $\Sigma^0\times_\varphi\Ss^1$. Moreover, the resulting submersion
\[  q:\G\to \Ss^1, \]
is a fibration of Lie groupoids.
\end{thm}
}

The short exact sequence  \eqref{eq:seq:cosymp:grpd} may fail to be smooth and, if smooth, it may fail to split. However, at the infinitesimal level it always splits. In fact, the Lie algebroid $A\to M$ of a cosymplectic groupoid $(\G,\omega,\al)$ carries a closed IM 2-form $\mu:A\to T^*M$ and a closed IM 1-form  $\nu:A\to M\times\RR$, corresponding to $\omega$ and $\al$, respectively. We then obtain the following:

\begin{thm}
\label{thm:main:algbrd}
If $(\G,\omega,\al)$ is a cosymplectic groupoid, its Lie algebroid $(A,\mu,\nu)$ is canonically isomorphic to the trivial central extension of the cotangent algebroid associated with the base Poisson manifold $(M,\pi_M)$:
\[ (A,\mu,\nu)\simeq (T^*M\oplus\RR,\pr_{T^*M},\pr_{M\times\RR}). \]
\end{thm}

Notice that given a source connected cosymplectic groupoid $(\G,\omega,\alpha)$, its source 1-connected cover is a cosymplectic groupoid. Applying the previous result, the latter are easy to describe:

\begin{cor}
\label{cor:integration}
If $(\G,\omega,\alpha)$ is a source 1-connected cosymplectic groupoid then there is a canonical isomorphism
\[ (\G,\omega,\alpha)\cong (\Sigma(M)\times\RR,\pr_\Sigma^*\Omega,\pr_\RR^*\d t) \]
where $(\Sigma(M),\Omega)$ is the source 1-connected symplectic integration of $(M,\pi_M)$.
\end{cor}

In the last part of this note we discuss how far Poisson groupoids and oversymplectic groupoids are from being cosymplectic groupoids.

We say that a Poisson groupoid $(\G,\pi_\G)$ is of corank 1 if its Poisson structure has constant rank equal to $\dim\G-1$, i.e., if its symplectic foliation is regular of codimension 1. We have the following simple criteria:

\begin{prop}
\label{prop:intro:Poisson:grpd}
Let $(\G,\pi_\G)$ be a Poisson groupoid of corank 1. Then $\pi_\G$ is associated with a cosymplectic structure if and only if there exists a Poisson vector field $E\in\X(\G)$ transverse the symplectic foliation of $\pi_\G$ which is bi-invariant.
\end{prop}

One the other hand, when $(\G\tto M,\pi_\G)$ is \emph{proper} of corank 1 with $\dim\G=2\dim M+1$ we will see that if its leafwise symplectic form admits a multiplicative extension (closed or not) then, up to a cover, it is homotopic to a cosymplectic groupoid through a homotopy that does not change the base Poisson structure:

{
\begin{thm}
\label{thm:intro:proper:Poisson:grpd}
Let $(\G,\pi_\G)\tto (M,\pi_M)$ be an orientable proper Poisson groupoid of corank 1 with $\dim\G=2\dim M+1$ and assume that there exists a multiplicative 2-form extending its leafwise symplectic form. If $(\widetilde{\G},\widetilde{\pi_\G})$ is its universal covering groupoid, then there is a path of Poisson structures $\widetilde{\pi}_\G^{\,t}\in\X^2(\widetilde{\G})$, starting at $\widetilde{\pi}_\G^{\,0}=\widetilde{\pi_\G}$, with the following properties:
\begin{enumerate}[(i)]
\item each $\widetilde{\pi}_\G^{\,t}$ is multiplicative of corank 1;
\item the Poisson structure on $M$ induced by $\widetilde{\pi}_\G^{\,t}$ is $\pi_M$;
\item $\widetilde{\pi}_\G^{\,1}$ is associated with a multiplicative cosymplectic structure.
\end{enumerate}
\end{thm}

Let us turn now to oversymplectic groupoids. Given such a groupoid $(\G,\omega)$, if the foliation given by $\ker\omega$ is simple then the leaf space 
\[ \Sigma:=\G/\ker\omega\]
is automatically a symplectic groupoid (this is the origin of the term ``oversymplectic''; see \cite{BCWZ}).  If  $(\G,\omega)$ is proper and $\ker\omega$ is an orientable line bundle, the quotient map $\Phi:\G\to\Sigma$ yields a short exact sequence of Lie groupoids
\[ 
\xymatrix{1\ar[r]& M\times\Ss^1\ar[r] & \G\ar[r]^{\Phi} & \Sigma \ar[r] &1}.
\]
It follows from recent results in \cite{FM22} that associated to such a sequence there is a well-defined \emph{multiplicative Chern class}, living in the multiplicative de Rham cohomology of $\Sigma\tto M$,
\[ c(\G) \in H^2_M(\Sigma). \]
We will show that:

\begin{thm}
\label{thm:intro:oversymplectic:grpd}
Let $(\G,\omega)$ be a corank 1, orientable, proper oversymplectic groupoid. If $\ker\omega$ is a simple foliation, then there exists $\al\in\Omega^1(\G)$ such that $(\G,\omega,\al)$ is a cosymplectic groupoid if and only if the multiplicative Chern class vanishes.
\end{thm}
}

This paper is organized as follows. In Section \ref{sec:groupoids} we recall some basics about cosymplectic structures, we introduce cosymplectic groupoids, establish its basic properties, and we prove Theorems \ref{thm:main:grpd} and \ref{thm:main:grpd:proper}. In Section \ref{sec:algebroids} we construct the infinitesimal data associated with a cosymplectic groupoid and we show that its Lie algebroid fits into a split short exact sequence, proving Theorem \ref{thm:main:algbrd} and Corollary \ref{cor:integration}. In Section \ref{sec:Poisson:grpd} we discuss relationships between Poisson groupoids, oversymplectic groupoids and cosymplectic groupoids, deducing in particular Proposition \ref{prop:intro:Poisson:grpd} and Theorems \ref{thm:intro:proper:Poisson:grpd} and \ref{thm:intro:oversymplectic:grpd}. We mostly follows the conventions and notation of the monograph \cite{CFM21}, to which we refer for background on Poisson structures and symplectic groupoids.

\bigskip
\emph{Acknowledgments.} We would like to thank 
the anonymous referee for pointing out an error in a first version of the paper.

\section{Cosymplectic groupoids}                  %
\label{sec:groupoids}                                       %

\subsection{Background on cosymplectic structures}            %
\label{sec:background}             %

A \emph{cosymplectic structure} on a manifold $Q$ of dimension $2m+1$ is a pair $(\omega, \al )$, where
$\omega$ is a closed 2-form, $\al$ is a closed 1-form and  $\omega ^m\wedge \al$ is a volume form. These structures were first introduced by Liberman \cite{Libermann59}. We collect here some basic facts about cosymplectic structures.

Associated with a cosymplectic  $(\omega,\al)$ on $Q$ there is a non-vanishing vector field $E\in\X(Q)$, called the \emph{Reeb vector field}, 
characterized by 
\begin{equation}
\label{eq:Reeb}
i_E\omega =0, \quad \al(E) =1.
\end{equation}
On the other hand, $\ker\al\subset TQ$ is an integrable distribution and the restriction of $\omega$ to its leaves is symplectic. The resulting symplectic foliation determines a regular \emph{Poisson structure} $\pi_Q\in\X^2(Q)$ of corank 1. 
Notice that, by construction, the closed 2-form $\omega$ extends the leafwise symplectic form of $\pi_Q$
\[ \omega(\pi_Q^\sharp(\be_1),\pi_Q^\sharp(\be_2))=\langle \be_1,\pi_Q^\sharp(\be_2)\rangle,\quad  (\be_1,\be_2\in T^*Q). \]
Moreover, the Reeb vector field is a Poisson vector field everywhere transverse to the symplectic foliation. 

Conversely, assume that $(Q,\pi_Q)$ is a regular Poisson structure of corank 1. If $E$ is a vector field transverse to the symplectic foliation, then one obtains
\begin{enumerate}[(i)]
\item a 2-form $\omega$ extending the leafwise symplectic form such that $\ker\omega=\langle E\rangle$ and 
\item a 1-form $\al$ such that $\al(E)=1$ and $\ker\al=\im\pi_Q^\sharp$. 
\end{enumerate}
It is not hard to check that $\omega$ and $\al$ are closed iff $E$ is a Poisson vector field, so we have (\cite[Proposition 18]{GMP}):

\begin{prop}
A regular Poisson structure $\pi_Q\in\X^2(Q)$ of corank 1 is defined by a cosymplectic structure if and only if there a Poisson vector field transverse to the symplectic foliation.
\end{prop}

Two cosymplectic structures $(\omega,\al)$ and $(\widetilde{\omega},\widetilde{\alpha})$ define the same Poisson structure $\pi_Q$ if and only $\widetilde{\alpha}=f\alpha$ for a nowhere vanishing Casimir function $f\in C^\infty(Q)$, and $\widetilde{\omega}-\omega$ is a closed two form vanishing on $\ker\al=\ker\widetilde{\al}$. In this case, the corresponding Reeb vector fields are related by $\widetilde{E}=\frac{1}{f} E$. 

The following examples give some basic constructions of cosymplectic manifolds related with symplectic manifolds:

\begin{ex}
If $(S,\omega_S)$ is a symplectic manifold then $Q=S \times \RR$ admits the  cosymplectic structure $(\pr_S^*\omega_S,\pr_\RR^*\d t)$, where $t$ denotes the coordinate on the second factor. In this case, the vector field $E$ is just $\frac{\partial}{\partial t}$, while the Poisson structure is $\pi_Q=\omega^{-1}\oplus 0$. 

Obviously, one can replace $\RR$ by $\Ss^1$ and $\d t$ by $\d\theta$, obtaining a cosymplectic structure on $S\times\Ss^1$. More generally, one can consider a principal $\Ss^1$-bundle $p:Q\to S$ over a symplectic manifold $(S,\omega_S)$ that admits a flat connection 1-form $\al\in\Omega^1(Q)$. Then $(p^*\omega_S,\al)$ is a cosymplectic structure whose underlying Poisson structure has symplectic foliation the horizontal foliation of $\al$. The Reeb vector field is the infinitesimal generator of the $\Ss^1$-action.
\end{ex}

\begin{ex}\label{ex:symp:mapping:torus}
Let  $(S,\omega_S)$ be a symplectic manifold and  $\varphi \colon S\to S$ a symplectomorphism. Recall that the corresponding \emph{symplectic mapping torus} is the fiber bundle
\[ q\colon S_\varphi \to \Ss^1,\] 
where $S_\varphi=(S\times\RR)/\Zz$ is the orbit space of the free and proper action
\[ \Zz\ract S\times\RR, \quad n\cdot(x,t)=(\varphi^n(x),t+n). \]
The manifold $S_\varphi$ is equipped with the cosymplectic structure $(\omega ,\al)$, where $\omega$ is the 2-form obtained from the basic form $\pr_S^*\omega_S\in\Omega^2(S\times\RR)$ and $\al =q^\ast (\d \theta )$, with $\theta$ the angle coordinate on $S^1$. The Reeb vector field $E\in\X(S_\varphi)$ is obtained by projecting the vector field $\frac{\partial}{\partial t}\in \mathfrak{X}(S\times \RR)$.
\end{ex}

\begin{rem}
Tischler's theorem \cite{Tischler70} shows that given a nowhere vanishing closed 1-form $\al$ on a compact manifold $Q$, there exists a fibration $q:Q\to\Ss^1$ with the property that one can choose $c>0$ such that $c\al$ and $\widetilde{\al}:=q^*\d\theta$ are $C^\infty$-close. Therefore, if $(Q,\omega,\al)$ is a compact cosymplectic manifold, one finds:
\begin{enumerate}[(i)]
\item The Poisson structure defined by $(\omega ,\al)$ is homotopic to the Poisson structure defined by $(\omega,q^*\d\theta)$;
\item $q:Q\to\Ss^1$ can be realized as a symplectic mapping torus with associated cosymplectic structure $(\widetilde{\omega},q^*\d\theta)$ for a modified closed 2-form (see, e.g., \cite{L}).
\end{enumerate} 
In this sense, a compact cosymplectic structure $(Q,\omega,\al)$ is, essentially, a symplectic mapping torus. 
\end{rem}

\begin{ex}
Let $(S,\Omega )$ be a symplectic manifold and $\iota: Q\hookrightarrow S$ a submanifold.  If $X$ is a symplectic vector field everywhere transverse to $Q$, then $(\iota ^*\Omega, \iota ^*(i_X\Omega ))$ defines a cosymplectic structure on $Q$. If $f$ is a function locally defining $Q$ and such that $X(f)=1$, then the Reeb vector field is given by $E=X_f|_Q$, where $X_f$ is the hamiltonian vector field of $f$.  The associated Poisson structure $\pi_Q\in\X^2(Q)$ is
\[ \pi_Q(\beta_1,\beta_2):=\omega^{-1}(\widetilde{\beta}_1,\widetilde{\beta}_2), \quad (\beta_1,\beta_2\in T^*Q), \]
where $\widetilde{\beta}\in T^*S$ denotes the unique extension of $\beta\in T^*Q$ satisfying $\beta(X)=1$.

Every cosymplectic manifold $(Q,\omega,\al)$ can be realized as a submanifold of a symplectic manifold $(S,\Omega)$: one lets $S=Q\times\RR$ with symplectic form $\Omega=\omega+\al\wedge\d t$. 
\end{ex}


\subsection{Cosymplectic groupoids}                  %
\label{subsec:groupoids}                                       %

Let $\G$ be a Lie groupoid over $M$. We denote by $\s$ and $\t$
the source and target maps, by $\m:\G^{(2)}\to\G$ the
multiplication (defined on the space $\G^{(2)}$ of pairs of
composable arrows), by $\textbf{i}:\G\to\G$ the inverse map, and
by $\eps:M\to \G$ the identity section. Our convention for the
groupoid multiplication is such that, given two arrows $x,y\in\G$,
the product $x\cdot y:=\m(x,y)$ is defined provided $\s(x)=\t(y)$.
Also, if $m\in M$ we write $1_m:=\eps(m)$ for the unit arrow over
$m$, and if $x\in\G$ we write $x^{-1}:=\textbf{i}(x)$ for the
inverse arrow. We denote the groupoid by $\G\tto M$. 

Recall that a form $\omega\in\Omega^k(\G)$ is said to be \emph{multiplicative} if 
\begin{equation}
\label{eq:mult:symp}
\m^*\omega=\pi_1^*\omega+\pi_2^*\omega,
\end{equation}
where $\pi_i:\G^{(2)}\to\G$ are the projections on each factor.

\begin{defn}
A \emph{cosymplectic groupoid} is a triple $(\G ,\omega, \al )$ where
$\G$ is a Lie groupoid and $(\omega ,\al )$ is a cosymplectic structure
on $\G$ with $\omega$ and $\al$ multiplicative forms.
\end{defn}

The following proposition gives some basic properties of a cosymplectic groupoid. It maybe useful to recall that a multiplicative distribution in a groupoid $\G\tto M$ is a distribution $D\subset T\G$ with $\d\s(D)=\d\t(D)=D_0$ and such that $D\tto D_0$ is a subgroupoid of the tangent groupoid $T\G\tto TM$. We refer to \cite{JO14} for basic facts about multiplicative distributions.

\begin{prop}\label{prop:Reeb:right:left}
Let $(\G ,\omega ,\al )$ be a cosymplectic groupoid over $M$. Then $\dim\G=2\dim M+1$ and one has:
\begin{enumerate}[(i)]
\item $\ker\al\subset T\G$ is a multiplicative distribution;
\item the induced Poisson structure $\pi_\G\in\X^2(\G)$ is multiplicative;
\item the Reeb vector field $E\in \X(\G)$ is a bi-invariant Poisson vector field. 
\end{enumerate}
\end{prop}

\begin{rem}
One can define multiplicative $k$-vector fields for any natural number $k$ \cite{BuCa}. We will only be interested in the case of a bi-vector field $\pi_\G\in\X^2(\G)$ for which multiplicativity amounts to requiring the map $\pi_\G^\sharp:T^*\G\to T\G$ to be a Lie groupoid morphism \cite{Mc2}. This is the analogue of the fact that a 2-form $\omega\in\Omega^2(\G)$ is multiplicative if and only if the map $\omega^\flat:T\G\to T^*\G$ is a Lie groupoid morphism. 
\end{rem}

\begin{proof}
It follows easily from the multiplicativity condition that for a cosymplectic groupoid $(\G ,\omega, \al )$ one has
\[ \textbf{i}^*\omega=-\omega, \quad \textbf{i}^*\al=-\al,\quad \textbf{i}_*E=-E,\]
where $E$ denotes the Reeb vector field. In particular, the pull backs of $\omega$ and $\al$ along the identity section vanish and the Reeb vector field is transverse to it. It follows that the subspace generated by the tangent space to the identity section and the Reeb vector field is maximally isotropic for $\omega$. Since $\omega$ has corank 1, we conclude that for a cosymplectic groupoid $\G\tto M$ one must have $\dim\G=2\dim M+1$. 

The kernel of a multiplicative form is a multiplicative distribution (see {{\cite{CrSaStru}}}). Hence (i) follows. Since $\im(\pi_\G)^\sharp=\ker\al$ and 
the symplectic forms on the leaves of $\pi_\G$ are obtained by restricting the multiplicative form $\omega$, we must have $\pi_\G$ multiplicative and (ii) also follows.

From the multiplicativity of $\omega$ we also have that
\[ D:=\langle E\rangle=\ker\omega, \] 
is a multiplicative distribution, i.e., $D$ is a subgroupoid of $T\G$. Since $E$ is transverse to the identity section, we find that
$D_0=D\cap TM=0_M$, i.e., we have $D\subset\ker\d \s\cap\ker\d \t$. This means that $D$ is a distribution which is both left and right invariant. In order to conclude that $E$ itself is both left and right invariant, notice that the right invariant vector field
\[ \der{e}:\G\to T\G,\quad g\mapsto d_{1_{\s(g)}} R_g (E_{1_{\s(g)}}),\] 
takes values in $D$, so $i_{\der{e}}\omega=0$. On the other hand, using the multiplicativity of $\al$, we obtain
\begin{align*}
\al _{g}(\der{e})&=\al _{g} (\d \m(E_{1_{\s(g)}},0_g))\\
  &=\al _{1_x} (E_{1_{\s(g)}})+\al_{g}(0_g)=1.
\end{align*}
Hence, by uniqueness, we must have $E=\der{e}$. Similarly, we find that $E=\esq{e}$, so $E$ is both left and right invariant and (iii) follows.
\end{proof}

Let $(\G,\omega,\al)$ be a cosymplectic groupoid with underlying multiplicative foliation $\F_{\G}=\ker\al$. From \eqref{eq:mult:symp}
one can prove that $\epsilon ^*\alpha=0$, which means 
that the identity section is tangent to this foliation, i.e.,
\[ TM\subset \ker\al. \]
Since we assume that $M$ is connected, it follows that it is contained in a single symplectic leaf $\Sigma^0$ of $\F_\G$. In general, this leaf is only an immersed submanifold. Still, denoting its symplectic form by $\omega_{\Sigma^0}=\omega|_{\Sigma^0}$, we have:

\begin{prop}
\label{prop:symplectic:leaf}
Let $(\G,\omega,\al)$ be a cosymplectic groupoid. Then $\Sigma^0$ is a Lie subgroupoid of $\G\tto M$ and $(\Sigma^0, \omega_{\Sigma^0})\tto M$ is a symplectic groupoid integrating $(M,\pi_M)$. In particular, $(M,\pi_M)$ is an integrable Poisson manifold.
\end{prop}

\begin{proof}
The condition that $\F_{\G}$ is multiplicative amounts to the identities
\begin{equation}
\label{eq:multiplicative:distr} 
\textbf{i}_*\F_{\G}=\F_{\G},\quad \m_*(\F_{\G}\timesst \F_{\G})=\F_{\G}, 
\end{equation}
where $\F_{\G}\timesst \F_{\G}=T\G^{(2)}\cap (\F\times\F)$ (note that since $\F+\ker\d\s=\F+\ker\d\t$ this intersection is transverse). Since $M$ is contained in the leaf $\Sigma^0$, it follows that the restriction of $\s$ and $\t$ to $\Sigma^0$ are surjective submersions. 

Now observe that inversion fixes the identity section, so the first condition in \eqref{eq:multiplicative:distr} implies that inversion maps leaves of $\F_{\G}$ to leaves of $\F_{\G}$. Since inversion fixes the identity section, it follows that it maps $\Sigma^0$ into itself. 

Similarly, the second condition in \eqref{eq:multiplicative:distr} implies that multiplication maps leaves of $\F_{\G}^{(2)}:=\F_{\G}\timesst \F_{\G}$ to leaves of $\F_{\G}$. Since $\Sigma^0\timesst \Sigma^0$ is a leaf of $\F_{\G}^{(2)}$ containing all the pairs $(1_x,1_x)$, it follows that $\Sigma^0$ is closed under multiplication.

Therefore we have smooth maps $\textbf{i}:\Sigma^0\to \G$ and $\m:\Sigma^0\timesst\Sigma^0\to\G$ with image lying in $\Sigma^0$. The fact that these are also smooth as maps into $\Sigma^0$ follows from the general fact that leaves of foliations are regularly immersed submanifolds.

The inclusion $\Sigma^0\hookrightarrow \G$ is a groupoid morphism, so obviously $ \omega_{\Sigma^0}:=\omega|_{\Sigma^0}$ is a multiplicative symplectic form, so $(\Sigma^0, \omega_{\Sigma^0})\tto M$ is a symplectic groupoid. The composition of this inclusion with the target of $\G$ gives a Poisson map 
\[ \t:(\Sigma^0, \omega_{\Sigma^0})\to (M,\pi_M),\] 
so this symplectic groupoid integrates $(M,\pi _M)$.
\end{proof}

\subsection{Examples of cosymplectic groupoids}                                        %
\label{subsec:groupoids:examples}                                       %
The basic examples of cosymplectic structures mentioned in Section \ref{sec:background}  all have multiplicative versions, yielding examples of cosymplectic groupoids.

\begin{ex}
\label{ex:trivial:product}
Let $(\Sigma,\Omega_\Sigma)$ be a symplectic groupoid over $M$. Then we can form the trivial abelian extension
\[ \xymatrix{1\ar[r]& M\times\RR\ar[r] & \Sigma\times\RR \ar[r] & \Sigma \ar[r] &1}\]
and equip $\G=\Sigma\times\RR$ with the cosymplectic structure $(\pr_\Sigma^*\Omega_\Sigma,\pr_\RR^*\d t)$, where $t$ denotes the coordinate in the second factor. This gives a cosymplectic groupoid, called the \emph{trivial central extension} of the symplectic groupoid $(\Sigma,\Omega_\Sigma)$ by $\RR$. 

A similar construction holds with $\RR$ replaced by $\Ss^1$ and $\d\theta$ instead of $\d t$. More generally, one can
consider a central extension of Lie groupoids with trivial kernel
\[ \xymatrix{1\ar[r]& M\times\Ss^1\ar[r] & \G\ar[r] & \Sigma \ar[r] &1} \]
where $(\Sigma,\Omega_\Sigma)$ a symplectic groupoid. A multiplicative Ehresmann connection for this extension is specified by a multiplicative 1-form $\al\in\Omega^1(\G,\RR)$ (see \cite{FM22,LSX09}). This connection is flat if and only if the form is closed, and will see in Section \ref{sec:oversymplectic} that in this case we obtain a multiplicative cosymplectic structure $(\pr_\Sigma^*\Omega_\Sigma,\al)$ in $\G$.
\end{ex}

\begin{ex}
\label{ex:mapping:torus}
Let $(\Sigma \tto M ,\omega )$ be a symplectic groupoid and $\varphi:\Sigma\to\Sigma$ a symplectomorphism satisfying
\begin{align*}
\s \circ \varphi &= \s, \qquad \t \circ \varphi  = \t, \\
\varphi ( gh) &= g\varphi  (h)=\varphi (g)h, \qquad (g,h)\in \Sigma ^{(2)}.
\end{align*}
for all $(g,h)\in \Sigma ^{(2)}$. Notice that these properties are satisfied by the time-one map of any bi-invariant vector field on a Lie groupoid (e.g., the Reeb vector field of a cosymplectic groupoid). These properties ensure that the map
\[
M\times \Zz\hookrightarrow \Sigma\times\RR,\quad (x,n)\mapsto (\varphi^n(1_x),n),
\]
make the trivial bundle of groups $M\times\Zz\tto M$ a closed, normal, subgroupoid inside the isotropy of the direct product groupoid $\Sigma\times\RR\tto M$. It follows that the mapping torus
\[ \Sigma\times_\varphi\Ss^1:= (\Sigma \times \RR)/\Zz,  \]
has a unique Lie groupoid structure making the following sequence of Lie groupoids exact
\[ \xymatrix{ 1\ar[r] & M\times \Zz\ar[r] & \Sigma\times\RR\ar[r] &  \Sigma\times_\varphi\Ss^1 \ar[r] & 1}.\]
Since $\Sigma\times_\varphi\Ss^1$ is a symplectic mapping torus, it has a cosymplectic structure which one checks is multiplicative. Hence, it is a cosymplectic groupoid. Notice that the map
\[  \Sigma\times_\varphi\Ss^1 \to \Ss^1,\quad  \left[g,t \right]\mapsto e^{2\pi i t}, \]
is a fibration of Lie groupoids.
\end{ex}

\begin{ex}
Let $(\Sigma,\Omega)$ be a symplectic groupoid over $M$ and let $\iota:\G\hookrightarrow \Sigma$ be a Lie subgroupoid. Assume that there exists a multiplicative symplectic vector field $X\in\X(\Sigma)$ transverse to $\G$. Then $\omega:=\iota^*\Omega$ and $\al:=\iota^*(i_X\Omega)$ define a multiplicative cosymplectic structure on the groupoid $\G$. 

Conversely, given a cosymplectic groupoid $(\G,\omega,\alpha)$ we can form a symplectic groupoid $(\Sigma,\Omega)$ so that $(\G,\omega,\alpha)$ is obtained from $(\Sigma,\Omega)$. We let $\Sigma$ be the product of the groupoid $\G\tto M$ with the identity groupoid $\RR\tto\RR$, so $\Sigma$ is a groupoid with space of arrows $\G\times\RR$ and space of objects $M\times\RR$. The symplectic form on $\Sigma$ is given by $\Omega=\omega+\alpha\wedge\d t$. One checks easily that $\Omega$ is multiplicative and that $\frac{\partial}{\partial t}$ is a multiplicative symplectic vector field transverse to $\G\times \{0\}\cong \G$.
\end{ex}

\begin{ex} 
\label{ex:not:extension}
For a concrete example of a cosymplectic groupoid which is not a central extension, let $\Sigma=\Tt^n\times\RR^n\tto\RR^n$ be the trivial bundle of Lie groups with fiber the torus $\Tt^n$. Denoting by $(\theta^1,\dots,\theta^n)$ angle coordinates on the torus and $(x^1,\dots,x^n)$ linear coordinates on $\RR^n$, we let $\Omega:=\sum_{i=1}^n \d \theta^i\wedge\d x^i$.
This is a multiplicative form so $(\Sigma,\Omega)$ is a symplectic groupoid. Now fix some $a=(a_1,\dots,a_n)\in\RR^n$ with $||a||=1$. The vector field
\[ X:=\sum_{i=1}^n a_i\frac{\partial}{\partial x^i}, \]
is a symplectic vector field transverse to the subgroupoid $\G=\Tt^n\times M\tto M$ where $M$ is the hyperplane
\[ M=\{(x^1,\dots,x^n)\in\RR^n: \sum_{i=1}^n a_i x^i=0\}. \]
Moreover, $X$ is multiplicative since its flow is a 1-parameter group of automorphisms of $\Sigma$. Hence, we are in the situation of the previous example, so we obtain a cosymplectic groupoid $(\G,\omega,\alpha)$. In particular, we find that $\alpha=-\sum_{i=1}^n a_i\d\theta^i$.
The hamiltonian vector field $X_f$ associated with the function $f(x,\theta)=\sum_{i=1}^n a_i x^i$ satisfies $i_{X_f}\omega=\d f|_{T\G}=0$ and 
$\alpha(X_f)=-\sum_{i=1}^n a_i^2=-1$. Hence, we conclude that the Reeb vector field is
\[ E=-X_f|_\G=-\sum_{i=1}^n a_i\frac{\partial}{\partial \theta^i}. \]
It follows that the orbit space of $E$ is smooth if and only if $\Zz a$ defines a discrete subgroup of $\Tt^n$, i.e., if and only if the ratios $a_i:a_j$ are all rational. Therefore this yields examples of cosymplectic groupoids which are not central extensions.
\end{ex}

\subsection{Central extensions and cosymplectic groupoids}                                        %
\label{subsec:central:extension:groupoids}                                       %

Let $(\G ,\omega ,\al )$ be a cosymplectic groupoid over $M$. Let us denote by $\K\subset \G$ the collection of all orbits of the Reeb vector field $E$ which intersect the identity section of $\G$. We call $\K$ the \emph{kernel of the cosymplectic groupoid} $(\G ,\omega ,\al )$.

\begin{thm}
The kernel $\K$ of a cosymplectic groupoid $(\G ,\omega ,\al )$ is a bundle of abelian groups which fits into a short exact sequence of topological groupoids over the same base
\begin{equation}
\label{eq:seq:cosymp:grpd:2}
\xymatrix{1\ar[r]& \K\ar[r] & \G\ar[r] & \Sigma \ar[r] &1}
\end{equation}
where $\Sigma$ is the orbit space of the Reeb vector field. When this space is smooth, this is a short exact sequence of Lie groupoids and $(\Sigma,\Omega)$ is a symplectic groupoid for a unique symplectic structure making the projection $\G\to\Sigma$ a Poisson map.
\end{thm}

\begin{proof}
Since $E$ is non-vanishing and transverse to the identity section, it follows that $\K$ is a submanifold of $\G$. By Proposition \ref{prop:Reeb:right:left}, since $E$ is bi-invariant it follows that $\K$ is a Lie subgroupoid of $\G$, which is actually a bundle of Lie groups contained in the isotropy of $\G$. One can form the quotient groupoid $\Sigma=\G/\K$, which is a topological groupoid, giving the short exact sequence \eqref{eq:seq:cosymp:grpd:2}.

Notice that the quotient groupoid $\Sigma=\G/\K$ can be identified with the space of orbits of the $\RR$-action defined by the Reeb vector field. When this orbit space is a smooth manifold, the form $\omega$ is basic for the $\RR$-action so there is a unique symplectic 2-form $\Omega$ in $\Sigma$ such $p^*\Omega=\omega$, where $p:\G\to\Sigma$ is the projection. Since $\omega$ is  multiplicative, it follows that $\Omega$ is also multiplicative, so $(\Sigma,\Omega)$ is a symplectic groupoid. One checks easily that $p:(\G,\pi_\G)\to (\Sigma,\Omega)$ is a Poisson map. Since $p$ is a submersion, it follows that $\Omega$ is the unique symplectic form with this property. 
\end{proof}

We will call a cosymplectic groupoid $(\G ,\omega ,\al )$ a \emph{central extension Lie groupoid} whenever the orbit space of the Reeb vector field is smooth, so $\G$ fits into a short exact sequence of Lie groupoids with abelian one-dimensional kernel.

\subsection{Proper cosymplectic groupoids}
{
Recall that a Lie groupoid $\G\tto M$ is called \emph{proper} if its space of arrows is Hausdorff and the map $(\s,\t):\G\to M\times M$ is proper. In this sections we restrict our attention to proper cosymplectic groupoids. 

\begin{ex}
The cosymplectic groupoids arising as symplectic mapping torus, as in Example \ref{ex:mapping:torus}, are proper whenever one starts with a proper symplectic groupoid $(\Sigma \tto M ,\omega )$. In this case, the symplectic leaves of the resulting cosymplectic groupoid $(\G,\omega,\al)$ are the fibers of $q:\G\to \Ss^1$, hence are embedded submanifolds.

The cosymplectic groupoid $(\G,\omega,\al)$ constructed in Example \ref{ex:not:extension}, which is not a central extension, is also proper, being a bundle of compact Lie groups. However, in this example the symplectic leaf through the identity section $\Sigma^0$ is not embedded.
\end{ex}

It turns out that a proper cosymplectic groupoid is a symplectic mapping torus, as in Example  \ref{ex:mapping:torus}, if and only if  the symplectic leaf $\Sigma^0$ containing the identity section is embedded.   

\begin{thm}
Let $(\G,\omega,\al)$ be a proper, source connected, cosymplectic groupoid and assume that the symplectic leaf $\Sigma^0\subset \G$ is embedded. Then:
\begin{enumerate}[(i)]
\item there is a time $t_0$ such that the flow of the Reeb vector field at time $t_0$ maps $\Sigma^0$ to itself yielding a symplectomorphism
\[ \varphi:=\varphi^{t_0}_E:\Sigma^0\to\Sigma^0. \]
 \item $\G$ is isomorphic to the symplectic mapping torus $\Sigma^0\times_\varphi\Ss^1$ and the resulting submersion
\[  q:\G\to \Ss^1, \]
is a fibration of Lie groupoids.
\end{enumerate}
\end{thm}

\begin{proof}
The Reeb vector field is a complete Poisson vector field transverse to the sympletic leaves of $(\G,\pi_\G)$. Hence, for each fix $t$, its flow $\varphi^t_E:\G\to\G$ maps leaves to leaves. We claim that there exists some smallest $t_0>0$ such that $\varphi^{t_0}_E(\Sigma^0)=\Sigma^0$.

Since the Reeb vector field $E$ is both left and right invariant, it satisfies
\begin{align*}
\s \circ \varphi^t_E &= \s, \qquad \t \circ \varphi^t_E  = \t, \\
\varphi^t_E ( gh) &= g\varphi^t_E  (h)=\varphi^t_E (g)h, \qquad ((g,h)\in \Sigma ^{(2)}).
\end{align*}
It follows that the map
\begin{equation}
\label{eq:morphism:proper:thm}
\Phi:\Sigma^0\times \RR\to \G,\quad (g,t)\mapsto \varphi^t_E(g). 
\end{equation}
is a Lie groupoid morphism. Since $\Sigma^0$ is embedded, this map is also a local diffeomorphism and the image of $\Phi$ is open and closed in $\G$. Since $M$ is connected and $\G$ is source connected, we have that $\G$ is connected, so $\Phi$ is surjective. 

Now observe that, for each $x\in M$, the map $\Phi$ restricts to a Lie group map 
\[ \Phi_x:\RR\to \G_x, \quad t\mapsto \Phi(1_x,t), \]
whose image is closed. Since $\G$ is proper, isotropy groups are compact, so the image of this map is compact. Hence, there exists a first time $t_0>0$ such that
\[ \Phi(1_x,t_0)\in\Sigma^0\cap\G_x. \]
Since, for each $t$, $\varphi^t_E:\G\to\G$ maps leaves to leaves we conclude that 
\[ \varphi^{t_0}_E(\Sigma^0)=\Sigma^0, \]
and $t_0$ is the smallest positive real satisfying this property, proving our claim. 

\begin{lem}
\label{lem:aux:proper}
The morphism \eqref{eq:morphism:proper:thm} yields a short exact sequence of Lie groupoids:
\[ \xymatrix{ 1\ar[r] & M\times \Zz\ar[r] & \Sigma^0\times\RR\ar[r]^{\Phi} & \G \ar[r] & 1}\]
where the first map is $(x,n)\mapsto(\varphi^{nt_0}_E(1_x),-nt_0)$. In particular, the groupoid $\G$ is isomorphic to a mapping torus:
\[ \G\simeq  (\Sigma^0\times\RR)/\Zz, \]
where the $\Zz$-action is generated by $(g,t)\mapsto (\varphi^{t_0}_E(g),t-t_0)$.
\end{lem}

Assuming this lemma, it remains to prove is that $\Phi$ pulls back the cosymplectic structure $(\omega,\al)$ to the cosymplectic structure $(\pr^*_{\Sigma^0}\omega_{\Sigma^0},\pr^*_\RR\d t)$.
This follows because:
\begin{enumerate}[(a)]
\item $\Phi$ is a map of the underlying foliations;
\item The Reeb vector fields $\partial_t$ and $E$ are $\Phi$-related
\[ \d_{(g,t)}\Phi(\partial_t)=\left.\frac{\d}{\d s}\right|_{s=t}\varphi^{s}_E(g)=E|_{\Phi(g,t)}. \]
\end{enumerate}

In fact, from (b), we find that  
\[ i_{\partial_t}\Phi^*\omega=\Phi^* i_E\omega=0. \]
Since $\omega$ is closed, it follows that $\Phi^*\omega$ is basic relative to $\pr_{\Sigma^0}:\Sigma^0\times\RR\to \Sigma^0$. On the other hand, for the section $s:\Sigma^0\to \Sigma^0\times\RR$, $g\mapsto (g,0)$, we have
\[ s^*\Phi^*\omega=(\Phi\circ s)^*\omega=\omega_{\Sigma^0}, \]
so we conclude that
\[ \Phi^*\omega=\pr^*_{\Sigma^0}\omega_{\Sigma^0}. \]

Similarly, from (a), we find that for any tangent vector $(v,0)\in T(\Sigma^0\times\RR)$
\[ i_{(v,0)}\Phi^*\alpha=\Phi^*(i_{\d\Phi(v,0)}\alpha)=0. \]
Since $\al$ is closed, it follows that $\Phi^*\al$ is basic relative to $\pr^*_\RR:\Sigma^0\times\RR\to\RR$. But using (b) again we find
\[ i_{\partial_t}\Phi^*\al=\Phi^*i_E\al=1, \]
so we conclude that
\[ \Phi^*\al=\pr^*_\RR\d t. \]
\end{proof}

\begin{proof}[Proof of Lemma \ref{lem:aux:proper}]
Observe that for each $x$ there is a smallest positive integer $n_0$ such that
\[ \varphi^{n_0t_0}_E(1_x)=1_x. \]
Note that $n_0$ is independent of $x$. This follows, e.g., because $n_0$ is the order of the group 
\[ \Sigma^0\cap \Phi_x(\RR)=\{1_x,\varphi^{t_0}_E(1_x),\dots, \varphi^{(n_0-1)t_0}_E(1_x)\}, \]
and these groups form a Lie group bundle when $x$ vary in $M$. From this it follows also that
\[ g\in \Sigma^0,\ \varphi^{t}_E(g)=1_x\quad \Leftrightarrow \quad 
\left\{
\begin{array}{l}
g=   \varphi^{n t_0}_E(1_x)   \\
\hskip 1,3 in \text{for some }n,k\in\Zz, \\
 t=-nt_0+k(n_0t_0).  
\end{array}
\right.
 \]
Since $\varphi^{n t_0}_E(1_x) =\varphi^{(n-kn_0) t_0}_E(1_x)$ we conclude that the kernel of the morphism \eqref{eq:morphism:proper:thm} is
 \[ \Ker\Phi=\{ (\varphi^{m t_0}_E(1_x),-mt_0): m\in\Zz\}. \]
This proves the lemma and completes also the proof of the theorem.
\end{proof}
}

\section{The Infinitesimal picture}               %
\label{sec:algebroids}                                    %

\subsection{Infinitesimal data of a cosymplectic groupoid}                                        %
\label{subsec:infin:data}                                       %
Let $(\G,\omega,\alpha)$ be a cosymplectic groupoid, with Reeb vector field $E$ and associated Poisson structure $\pi_\G$. All these geometric structures have infinitesimal versions, as we now explain.

In general, will denote by $A$ a Lie algebroid with bundle projection $p:A\to M$, anchor $\rho_A:A\to TM$, and Lie bracket $[~,~]_A$ 
on its space of sections $\Gamma(A)$. 
Our conventions are such that if $\G\tto M$ is a Lie groupoid, then its Lie algebroid
$A=A(\G)$ has fiber $A_x:=\Ker\d_{1_x}\s$ and anchor $\rho_A|_x:=\d_{1_x}\t$. Moreover, its space of sections $\Gamma(A)$ is identified with the
space $\X_r(\G)$ of right invariant vector fields on $\G$ and we will denote by $\der{X}\in \X_r(\G)$ the right
invariant vector field corresponding to $X\in \Gamma(A)$. A multiplicative form $\omega\in\Omega^k(\G)$ induces 
a pair of bundle maps $\mu:A\to\wedge^{k-1}T^*M$, $\tilde{\mu}:A\to\wedge^k T^*M$, defined by
\begin{align*} 
\mu(a)(v_1,\dots,v_{k-1})&=\omega(a,\d\eps(v_1),\dots,\d\eps(v_{k-1})), \\
\tilde{\mu}(a)(v_1,\dots,v_{k})&=\d\omega(a,\d\eps(v_1),\dots,\d\eps(v_{k})).
\end{align*}
These maps satisfy the following conditions that characterize infinitesimal multiplicative (IM) forms (see, e.g., \cite{BuCa,BCWZ}):
\begin{align}
i_{\rho_A(b)}\mu (a)&=-i_{\rho_A(a)}\mu(b), \notag \\
\mu ([a,b]_A)&=\Lie _{\rho _A (a)}\mu(b)-i_{\rho _A(b)}(\d\mu (a)+\tilde{\mu}(a)),\label{eq:IM}\\
\tilde{\mu}([a,b]_A)&=\Lie _{\rho _A (a)}\tilde{\mu}(b)-i_{\rho _A(b)}\d\tilde{\mu} (a),\notag
 \end{align}
 for all sections $a,b\in \Gamma(A)$. For a closed IM form the component $\tilde{\mu}$ vanishes.

After these preliminaries we can now list the infinitesimal data corresponding to a cosymplectic groupoid. Let $(\G,\omega,\alpha)$ be a cosymplectic groupoid and denote by $A\to M$ its Lie algebroid. Then:
\begin{enumerate}[(i)]
\item The multiplicative closed 1-form $\al$ induces a closed IM 1-form $\nu:A\to\RR$;
\item The multiplicative closed 2-form $\omega$ induces a closed IM 2-form $\mu:A\to  T^*M$;
\item The Reeb vector field $E$ induces a central section $e\in\Gamma(A)$, i.e., $E=\der{e}=\esq{e}$ where $\rho(e)=0$ and $[e,a]=0$, for all $a\in\Gamma(A)$;
\item The multiplicative Poisson structure $\pi_\G$ induces a unique Poisson structure $\pi_M\in\X^2(M)$, for which the target is a Poisson map, and the source is an anti-Poisson map.
\end{enumerate}
Only the last item needs some justification. One can show directly from the condition that $\pi_\G$ is multiplicative that the Poisson bracket of functions locally constant on the $\t$-fibers is a function locally constant on the $\t$-fibers (see, e.g., \cite{We}), so that there is a unique Poisson structure on $M$ for which the submersion $\t:\G\to M$ is Poisson.

The infinitesimal data above has various relationships between themselves, which can be stated in a concise form as follows:

\begin{prop}
\label{prop:inf:data}
The Lie algebroid $A\to M$ of a cosymplectic groupoid $(\G,\omega,\al)$ is a central extension
\begin{equation}
\label{eq:seq:cosymp:algbrd:2}
\xymatrix{0\ar[r]& \kk \ar[r] & A\ar@{-->}@/^/[l]^{\underline{\nu}}\ar[r]^{\mu} & T^*M \ar[r] &0}
\end{equation}
where $T^*M$ is equipped with the cotangent Lie algebroid structure associated with the Poisson manifold $(M,\pi_M)$ and $\kk$ is the trivial line bundle generated by the central section $e\in\Gamma(A)$. This extension has a natural splitting given by:
\[ \underline{\nu}:A\to \kk,\quad a\mapsto \nu(a)e.\]
\end{prop}

\begin{proof} $\quad$ 

\vskip 5 pt
(i) \emph{$\mu: A\to T^*M$ is a surjective Lie algebroid map:} The definition of $\mu$ shows that we have a commutative diagram:
\[
\xymatrix{
T\G \ar@<.5ex>[d]\ar[r]^{\omega ^\flat }& T^*\G \ar@<.5ex>[d]\\
TM \ar@<.5ex>@{<-}[u] \ar[r]^{\mu^*} & \ar@<.5ex>@{<-}[u]  A^* }
\]
where $\mu^*$ is the transpose of $\mu$. By Proposition \ref{prop:Reeb:right:left}, we know that $E\in\ker\omega^\flat$ is transverse to the identity section, so we conclude that $\mu ^*$ is injective. 

Now observe that since $\t:(\G,\pi_\G)\to (M,\pi_M)$ is a Poisson map, using the definition of $\pi_\G$ and $\mu$, we obtain
\[ \pi_M^\sharp(\mu (a) )=\d_{1_x}\t\cdot \pi_\G^\#\cdot (\d_{1_x}\t)^*(\mu (a))=\rho _A(a)\]
where $a\in A_x$. On the other hand, since $\mu$ is a closed IM form, relations \eqref{eq:IM} give
\begin{align*}
\mu ([a,b]_A)&=\Lie _{\rho_A(a)}\mu(b)-i_{\rho_A(b)}\d\mu (a)\\
&=\Lie_{\pi_M^\sharp(\mu (a))}\mu(b)-i_{\pi_M^\sharp(\mu (b))}\d\mu (a)=[\mu(a),\mu(b)]_{\pi_M},
\end{align*}
so $\mu:A\to T^*M$ is a Lie algebroid morphism.

\vskip 5 pt
(ii) $\ker(\mu)= \langle e \rangle$: Since $\mu$ is surjective, its kernel is a rank 1 vector sub-bundle. Since $e\in\Gamma(A)$ is a non-vanishing section, all we have to check is that $\mu(e)=0$. This is clear from the definition of $\mu$ since $\mu(e)=(i_E\omega)|_{TM}=0$.

\vskip 5 pt
(iii) \emph{$\underline{\nu}:A\to \langle e\rangle$, $a\mapsto \nu(a)e$, splits the short exact sequence \eqref{eq:seq:cosymp:algbrd:2}:} Notice that we have $\nu(e)=1$, since
\[ \nu(e)(x)=\alpha_{1_x}(E_{1_x})=1. \]
This shows that $\underline{\nu}$ is a splitting as a short exact sequence of vector bundles. Associated with this splitting there is a $T^*M$-connection on the bundle $\kk$. Because $e$ commutes with any section of $A$,  we have that $e$ is a flat section
\[ \nabla_\be e=0. \]
On the other hand, the curvature 2-form of this splitting is given by
\[ c(\be_1,\be_2)=\nu([a_1,a_2]_A), \]
where $a_i\in \Ker \nu$ is the unique element such that $\mu(a_i)=\be_i$. But this curvature 2-form vanishes since $\nu$ is a closed IM form and from \eqref{eq:IM} find
\[
\nu([a_1,a_2]_A)=\Lie_{\rho_A(a_1)}\nu(a_2)-i_{\rho_A(a_2)}\d \nu (a_1)=0,
\]
whenever $a_1,a_2\in\Gamma(\Ker \nu)$.
\end{proof}

The previous proposition shows that the space of objects of a cosymplectic groupoid is a Poisson manifold, and that we have a canonical isomorphism:
\begin{equation}
\label{eq:iso:algebroid} 
A\cong T^*M\oplus \RR,\quad a\mapsto (\mu(a),\nu(a)).
\end{equation}
Under this isomorphism, the anchor becomes 
\[ \rho_A:T^*M\oplus\RR\to TM,\quad (\be,\lambda)\mapsto \pi^\sharp_M(\be ),\] 
while the bracket on sections $(\be_i, f_i)\in \Omega ^1(M)\times C^\infty (M)$ can be written as:
\[
[(\be_1 ,f_1),(\be_2,f_2)]_A = ([\be_1,\be_2]_{\pi_M},\pi_M^\sharp (\be_1)(f_2)-\pi_M^\sharp(\be_2)(f_1)).
\]
In other words, we have:

\begin{cor}
\label{cor:algebroid:trivial}
If $(\G,\omega,\al)$ is a cosymplectic groupoid, its Lie algebroid $(A,\mu,\nu)$ is canonically isomorphic via \eqref{eq:iso:algebroid} to the trivial central extension of the cotangent algebroid associated with the base Poisson manifold $(M,\pi _M)$:
\[ (A,\mu,\nu)\simeq (T^*M\oplus\RR,\pr_{T^*M},\pr_{M\times\RR}). \]
\end{cor}

The only missing piece on the infinitesimal side is what corresponds to the multiplicative Poisson structure $\pi_\G$. At the infinitesimal level such a structure corresponds to a \emph{Lie bialgebroid}. We recall that a Lie bialgebroid is a pair of Lie algebroid structures $(A,A^*)$, on a vector bundle $A\to M$ and on its dual bundle $A^*\to M$, such that for any $X, Y\in \Gamma(A)$,
\[
\d_* [X,Y]_A=[X , \d _* Y]_A-[Y ,\d _* X]_A.
\] 
where $\d_*$ denotes the $A^*$-differential (\cite{MX94}). Given a multiplicative Poisson structure $\pi _\G$ on a groupoid $\G$ with Lie algebroid $A$ one obtains a Lie algebroid structure on $A^*$, so that $(A,A^*)$ is a Lie bialgebroid. Namely, the $A^*$-differential is characterized by 
\[
\der{\d_* X}=-[\der{X} , \pi _\G],\qquad (X\in \Gamma (A)).
\]
The Lie bialgebroid of a cosymplectic groupoid is again rather special.

\begin{prop}
\label{prop:Lie:bialgebroid:cosymp}
If $(\G,\omega,\al)$ is a cosymplectic groupoid, \eqref{eq:iso:algebroid} gives an isomorphism of Lie bialgeboids
\[ (A,A^*)\simeq (T^*M\oplus\RR,TM\oplus \RR), \]
where $T^*M$ denotes the cotangent Lie algebroid of the base Poisson manifold $(M,\pi_M)$.
\end{prop}

\begin{proof}
By Corollary \ref{cor:algebroid:trivial}, we already know that the IM forms corresponding to $(\omega,\al)$ give a Lie algebroid isomorphism:
\[ A\cong T^*M\oplus \RR,\quad a\mapsto (\mu(a),\nu(a)). \]
On the other hand, the central section $e\in\Gamma(A)$ satisfies:
\[ \d_*e=0, \]
since the Reeb vector field $E=\der{e}$ is a Poisson vector field on $\G$ (see, \cite[Thm 11.4.7]{Mc2}). This implies that the transpose of the map $(\mu,\nu)$ is also a Lie algebroid isomorphism:
\[ TM\oplus \RR \cong A^*,\quad ({u},\lambda)\mapsto (\mu^*({u}),\nu^*(\lambda)). \]
\end{proof}

\subsection{Source 1-connected cosymplectic groupoids}                                        %
\label{subsec:1-connected}                                       %

By the results in the previous section, source 1-connected cosymplectic groupoids are very easy to describe:

\begin{thm}
The base of any cosymplectic groupoid $(\G,\omega,\alpha)$ is an integrable Poisson manifold $(M,\pi_M)$. If $\G$ is source 1-connected then there is a canonical isomorphism
\[ (\G,\omega,\alpha)\cong (\Sigma(M)\times\RR,\pr_\Sigma^*\Omega,\pr_\RR^*\d t) \]
where $(\Sigma(M),\Omega)$ the source 1-connected symplectic integration of $(M,\pi_M)$.
\end{thm}

\begin{proof}
Since $A\cong T^*M\oplus \RR$, $a\mapsto (\mu(a),\nu(a))$, is a Lie algebroid isomorphism, it follows that $A$ is integrable iff and only if $T^*M$ is integrable. When $\G$ is source 1-connected, the integration of this isomorphism gives the desired groupoid isomorphism. This groupoid isomorphism maps $\pr_\Sigma^*\Omega$ to $\omega$ and $\pr_\RR^*\d t$ to $\alpha$.
\end{proof}
In general, if $(\G,\omega,\alpha)$ is only source connected, we have 
\[ (\Sigma(M)\times\RR)/\Lambda,\] 
where $i:\Lambda\hookrightarrow \Sigma(M)\times\RR$ is an embedded bundle of discrete Lie groups such that $i^*\Omega=0$ and $i^*\d t=0$.  If we assume that the orbit space of the Reeb vector field $E\in\X(\G)$ is smooth, we obtain that $\G$ is a central extension of some  symplectic integration $\Sigma$ of $(M,\pi_M)$:
\[ \xymatrix{1\ar[r]& \K\ar[r] & \G\ar[r] & \Sigma \ar[r] &1}. \]
In general, this sequence fails to split. Moreover, it may have a groupoid splitting, so that $\G\cong \Sigma\times\K$, while the cosymplectic structure may not be the trivial one (as in Example \ref{ex:trivial:product}). This is illustrated in the next example. 

\begin{ex}
Let $\G=\RR^2\times \Ss^1\tto\RR$ be the trivial bundle of Lie groups with projection $(x,y,\theta)\mapsto x$ and fiber $\RR\times \Ss^1$.  The forms
\[ \omega:=\d x\wedge \d y+\d x\wedge\d\theta,\quad \al:=\frac{1}{2}\left(\d y-\d\theta \right),\]
define a multiplicative cosymplectic structure on $\G$. The corresponding Reeb vector field is
\[ E=\partial_y-\partial_\theta. \]
It follows that the kernel of this cosymplectic groupoid is a trivial bundle $\RR\times\RR\to\RR$ and the leaf space of this vector field can be identified with $\RR\times\Ss^1$, giving rise to the central extension:
\[ \xymatrix{1\ar[r]& \RR\times\RR\ar[r] & \RR^2\times\Ss^1\ar[r] & \RR\times\Ss^1 \ar[r] &1} \]
Here the first map is given by $(x,y)\mapsto (x,y,-y)$ while the second map is given by $(x,y,\theta)\mapsto (x,y+\theta)$. This sequence has the splitting $(x,\theta)\mapsto (x,0,\theta)$. 

We claim that although  $\G\cong \Sigma\times\K$ as a groupoid, the cosymplectic structure structure is not isomorphic to $(p_\Sigma^*\Omega,p_\K^*\d \theta)$.  In fact, the symplectic leaves of $\G$ are the leaves of the distribution $\d y-\d\theta=0$, so admit the parametrization $(x,y)\mapsto (x,y,y+c)$, with $c\in\Ss^1$. Hence, the symplectic leaves are diffeomorphic to $\RR^2$, while a trivial extension $\Sigma\times\K$ has symplectic leaves diffeomorphic to $\RR\times\Ss^1$.
\end{ex}

\section{Poisson groupoids}               %
\label{sec:Poisson:grpd}                                     %

\subsection{Poisson groupoids of corank 1}               %

Let $\G \tto M$ be a Lie groupoid equipped with a regular multiplicative Poisson structure $\pi _\G$ of corank 1. Recall that a Poisson structure $\pi _\G$ is multiplicative if and only if $\pi _\G^\sharp \colon T^*\G \to T\G$ is a morphism of Lie groupoids
\[
\xymatrix{
T^*\G \ar@<.5ex>[d]\ar[r]^{\pi_\G ^\sharp }& T\G \ar@<.5ex>[d]\\
A^* \ar@<.5ex>@{<-}[u] \ar[r]^{\rho_{A^*} } & \ar@<.5ex>@{<-}[u] TM }
\]
In this diagram $A^*$ is identified with the annihilator subbundle $(TM)^0\subset T^*_M\G$ and the anchor $\rho_{A^*}$ is the restriction of $\pi_\G ^\sharp$ to this subbundle.  The condition that $\pi _\G$ is regular of corank 1 places restrictions on the codimension of $M$ in $\G$.

\begin{prop}
\label{prop:dim}
Let $(\G \tto M,\pi _\G)$ be a regular Poisson groupoid of corank 1. Then either:
\[ \mathrm{(C1)}\quad \dim\G=2\dim M+1\qquad \textrm{or}\quad  \mathrm{(C2)}\quad \dim\G=2\dim M-1. \]
\end{prop}
\begin{proof}
Since $\pi_\G$ has constant rank, we can choose an extension $\omega\in \Omega ^2(\G)$ of the leafwise symplectic form, i.e.,
\[
\omega (\pi_\G^\sharp (\alpha),\pi _\G^\sharp (\beta ))=\langle \alpha ,\pi _\G^\sharp (\beta )\rangle.
\]
Note that $\omega$ is not assumed to be multiplicative. We claim that $-\textbf{i}^*\omega$ is also an extension of the leafwise symplectic form. Indeed, since $\textbf{i}$ is anti-Poisson, we find
\begin{align*}
-\textbf{i}^*\omega  (\pi_\G^\sharp (\alpha),\pi _\G^\sharp (\beta ))& = -\omega  (\d\textbf{i} \smalcirc \pi_\G^\sharp (\alpha),\d\textbf{i} \smalcirc \pi _\G^\sharp (\beta ))\\
&= -\omega  ( \pi_\G^\sharp (\textbf{i}^*(\alpha)), \pi _\G^\sharp (\textbf{i}^*(\beta ))) \\
&=  - \langle  \textbf{i}^*(\alpha), \pi _\G^\sharp (\textbf{i}^*(\beta ))) = \langle \alpha ,\pi _\G^\sharp (\beta )\rangle.
\end{align*}
It follows that $\tilde{\omega}=\frac{1}{2} \left ( \omega - \textbf{i}^* \omega \right )$ is also an extension satisfying additionally:
\[ \textbf{i}^*\tilde{\omega}=-\tilde{\omega}.\] 
Since $\textbf{i}\circ\epsilon=\epsilon$, we deduce that $\epsilon^*\tilde{\omega}=0$. In addition, $\rank \tilde\omega =\dim\G-1$ implies that 
\[ 2\dim M\le \dim\G+1. \]
Now observe that since $\pi_\G$ is multiplicative, the identity section is coisotropic, i.e., we have $\pi^\sharp_\G(TM)^0\subset TM$. So we also have 
\[ 2\dim M\ge \dim\G-1. \] 
Since $\dim \G$ is odd, it follows from these two estimates that $\dim\G=2\dim M\pm1$.
\end{proof}

Note that both cases listed in this proposition exist. For example, let $(\Sigma,\Omega)$ be a symplectic groupoid then:
\begin{enumerate}[(C1)]
\item the product of $\Sigma$ with a 1-dimensional Lie group, e.g. $\Ss^1\tto\{*\}$, is a Poisson groupoid $(\G \tto M,\pi _\G)$ of corank 1 with $\dim\G=2\dim M+1$;
\item the product of $\Sigma$ with a 1-dimensional identity groupoid, e.g. $\Ss^1\tto\Ss^1$, is a Poisson groupoid $(\G \tto M,\pi _\G)$ of corank 1 with $\dim\G=2\dim M-1$.
\end{enumerate}
We will be interested in Poisson groupoids of type (C1) since these are the ones that arise in connection with cosymplectic groupoids.


\begin{prop}
\label{prop:type:C1}
If $(\G \tto M,\pi _\G)$ is a Poisson groupoid of corank 1 then the following conditions are equivalent:
\begin{enumerate}[(i)]
\item $\dim\G=2\dim M+1$;
\item $\rank(\pi _\G) = 2 \dim M$;
\item $\rho_{A^*}:A^*\to TM$ surjective;
\item the identity section $M$ is contained in a symplectic leaf of $\pi_\G$.
\end{enumerate}
\end{prop}
\begin{proof} $\quad$

(i) $\Rightarrow$ (ii) Obvious.

(ii) $\Rightarrow$ (iii) Since $\dim \G=2\dim M+1$, we have $\rank(TM)^0=\dim M+1$. Hence, the map $\rho_{A^*}:=(\pi_\G ^\sharp)|_{(TM)^0}:(TM)^0\to TM$ must be surjective.

(iii) $\Rightarrow$ (iv) If (iii) holds, then $TM=\im(\rho_{A^*})\subset \im(\pi_\G ^\sharp)$.

(iv) $\Rightarrow$ (i) For any Poisson groupoid $(\G \tto M,\pi _\G)$ one has:
\[ \pi_\G ^\sharp(TM)^0\subset TM,\quad \pi_\G ^\sharp(\ker\d\s)^0\subset \ker\d\t. \]
If (iv) holds, then we must have $\rank (TM)^0\ge \rank(TM)$, so case (C2) is excluded.
\end{proof}

In the sequel, by a ``Poisson groupoid of type (C1)'' we mean a Poisson groupoid of corank 1 satisfying any of the equivalent conditions of Proposition \ref{prop:type:C1}.

\begin{cor}
Let $(\G,\pi_\G)$ be a Poisson groupoid of type (C1). Then the symplectic leaf $\Sigma^0$ of $\pi_\G$ containing the identity section is a (symplectic) subgroupoid of $(\G,\pi _\G)$ integrating $(M,\pi _M)$. In particular, $(M,\pi _M)$ is an integrable Poisson manifold.
\end{cor}

\begin{proof}
The transpose $\rho_{A^*}^*: T^*M\to A$ is a Lie algebroid morphism, where $T^*M$ is the cotangent bundle algebroid associated with the base Poisson manifold $(M,\pi_M)$ (see \cite{Mc2}). By Proposition \ref{prop:type:C1} (iii),  this morphism is injective, so $T^*M$ is a Lie subalgebroid of the integrable algebroid $A$, and the result follows. 
Alternatively, one can also apply the argument in the proof of Proposition \ref{prop:symplectic:leaf}.
\end{proof}

\subsection{Poisson groupoids vs cosymplectic groupoids}               %

When is a Poisson groupoid of corank 1 a cosymplectic groupoid? By Proposition \ref{prop:Lie:bialgebroid:cosymp}, its Lie bialgebroid must be a central extension. A necessary and sufficient condition is the following multiplicative version of the criteria for a Poisson manifold of corank 1 to be cosymplectic:

\begin{prop}
\label{prop:Poisson:grpd}
A Poisson groupoid $(\G,\pi_\G)$ is cosymplectic if and only if $\pi_\G$ is regular of corank 1 and there exists a non-vanishing, bi-invariant, Poisson vector field $E\in\X(\G)$ transverse to the  symplectic foliation.
\end{prop}

\begin{proof}
In one direction, we already know that the Reeb vector field of a cosymplectic groupoid is a Poisson vector field transverse to the symplectic foliation, which is both left and right invariant.

To prove the reverse direction, assume that  $(\G,\pi_\G)$ is a Poisson groupoid of corank 1 that admits a non-vanishing Poisson vector field $E\in\X(\G)$ transverse to the symplectic foliation, which is both left and right invariant. We extend the symplectic forms on the leaves to a 2-form $\omega$ by requiring $i_E\omega=0$. Also, we define a 1-form $\alpha$ by requiring $\alpha(E)=1$ and $\ker\al=\im\pi^\sharp_\G$. Since $E$ is a Poisson vector field, one checks easily that $\omega$ and $\alpha$ are closed. We claim that $\omega$ and $\alpha$ are multiplicative.

Take $(X_1,X_2)\in T\G^{(2)}$. Since $E$ is tranverse to the symplectic foliation $X_i=\lambda _iE+\pi _\G ^\sharp (\gamma ^i)$, $i=1,2$. Moreover, the fact that $E$ is left and right invariant implies that $\d \s ( \pi _\G ^\sharp (\gamma_1) )=\d \t ( \pi _\G ^\sharp (\gamma_2) )$. Since $\pi_\G$ is multiplicative, $\im \pi_\G ^\sharp\subseteq T\G$ is a multiplicative distribution, so we find 
\[
\d \m ( X_1 ,X_2 )= (\lambda _1 +\lambda _2) E + \pi _\G^\sharp (\gamma_1\cdot \gamma _2).
\]
Using that $\alpha (E)=1$, $\Ker \alpha =\im \pi_\G^\sharp$, and that $\omega$ is an extension of $\pi _\G$ with $i_E\omega=0$, we can conclude that $\omega$ and $\alpha$ are multiplicative. For instance, for $\alpha$,
\[
\m ^*\alpha (X_1,X_2)= \alpha (\d \m (X_1,X_2)) =  \alpha ((\lambda _1 +\lambda _2) E + \pi _\G^\sharp (\gamma_1\cdot \gamma _2)) = \lambda _1+\lambda _2.
\]
On the other hand, 
\[
( \pi _1 ^*\alpha + \pi _2 ^*\alpha )(X_1,X_2)= \alpha (\lambda _1E+\pi _\G ^\sharp (\gamma ^1))+  \alpha (\lambda _2E+\pi _\G ^\sharp (\gamma ^2))= \lambda _1+\lambda _2.
\]
Thus, $\alpha$ is multiplicative.
\end{proof}

\subsection{Proper Poisson groupoids of corank 1}               %

{ 

Proposition \ref{prop:Poisson:grpd} shows that given a Poisson groupoid $(\G,\pi_\G)$ of corank 1, in order to have a compatible multiplicative cosymplectic structure one needs a vector field $E\in\X(\G)$ transverse to symplectic foliation satisfying two properties:
\begin{enumerate}[(a)]
\item $E$ is bi-invariant;
\item $E$ is a Poisson vector field.
\end{enumerate}
We now analyze these properties for the important special case of \emph{proper} Poisson groupoids. 

First, as a general remark, note that the existence of a vector field transverse to the symplectic foliation means that this foliation is co-orientable. Since the leaves are oriented (being symplectic), this is equivalent to $\G$ being orientable. Hence, we will assume this condition throughout this discussion.

We start by looking into condition (a).

\begin{prop}
\label{prop:condition:a}
Let $(\G,\pi_\G)$ be an orientable Poisson groupoid of corank 1 with symplectic foliation $\F_{\pi_\G}$. The following two conditions are equivalent:
\begin{enumerate}[(i)]
\item There exists a bi-invariant vector field transverse to $\F_{\pi_\G}$;
\item There exists a multiplicative 1-form whose kernel is $\F_{\pi_\G}$;
\end{enumerate}
and they imply that
\begin{enumerate}[(i)]
\item[(iii)] There exists a multiplicative 2-form extending the leafwise symplectic form. 
\end{enumerate}
If $\G$ is proper and $\dim \G=2\dim M+1$ then the 3 conditions are equivalent.
\end{prop}

\begin{proof}
(i) $\Leftrightarrow$ (ii) A vector field $E\in\X(\G)$ transverse to $\F_{\pi_\G}$ determines a unique 1-form $\al\in\Omega^1(\G)$ by
\[ i_E\al=1,\quad \ker\al=T\F_{\pi_\G}. \]
Conversely, given $\al$ these conditions determine $E$. By an argument entirely similar to the last part of the proof of Proposition \ref{prop:Poisson:grpd}, one checks that $\al$ is multiplicative if and only if $E$ is both left and right invariant.

(i) $\Rightarrow$ (iii) Given a vector field $E\in\X(\G)$ transverse to $\F_{\pi_\G}$ which is both right and left invariant, we define an extension $\omega\in\Omega^2(\G)$ of the leafwise symplectic form $\omega_{\F_{\pi_\G}}$ by requiring
\[ i_E\omega=0. \]
Again, one checks easily that $\omega$ is multiplicative. 

(iii) $\Rightarrow$ (i) We assume now that $\G$ is proper and $\dim \G=2\dim M+1$. We let $\omega\in\Omega^2(\G)$ be a multiplicative 2-form extending the leafwise symplectic form. By \cite[Lemma 3.3]{BCWZ}, the multiplicativity of $\omega$ implies that at points $x\in M$ one has
\begin{align*}
\ker\omega_x&=(\ker\omega_x\cap \ker\d_x \s)\oplus (\ker\omega_x\cap T_x M)\\
&=(\ker\omega_x\cap \ker\d_x \t)\oplus (\ker\omega_x\cap T_x M).
\end{align*}
By Proposition \ref{prop:type:C1} (iv), we have $\ker\omega_x\cap T_x M=\{0\}$, so we conclude that:
\[ \mathfrak{k}:=(\ker\omega)|_M\subset  \ker\d_M \t\cap \ker\d_M \s=\ker\rho_A. \]
Since $\ker \omega$ is multiplicative, it is a distribution which is both right and left invariant and, as a consequence,
$\mathfrak{k}$ is invariant under the adjoint action of $\G$. In addition, since $\F_{\pi_\G}$ is co-orientable, there exists a non-vanishing section $\tilde{e}\in {\Gamma}(\mathfrak{k})$. The corresponding right-invariant vector field $\tilde{E}=\der{\tilde{e}}$ may fail to be left-invariant. To correct this we use properness of $\G$. Since $\mathfrak{k}$ is $\Ad$-invariant, the function $c\colon \G\to \RR ^+$ defined by
\[
\d L_g(\tilde{E}_{\s (g)})=c(g)\tilde{E}_g,
\]
is multiplicative:
\[ c(gh)=c(g)c(h), \quad  ((g,h)\in\G^{(2)}). \]
Since $\G$ is proper, there is function $f\colon M\to \RR^+$ such that
\[ c(g)=f(\t(g))/f(\s (g)). \] 
The section $e:= f\tilde{e}\in {\Gamma}(\mathfrak{k})$ gives the desired vector field $E:=\der{e}=\esq{e}$.
\end{proof}

We now turn to condition (b) assuming than condition (a) holds. We show that, up to a cover, a proper Poisson groupoid of type (C1) satisfying (a) is homotopic to one satisfying also (b) through a homotopy that does not change the Poisson structure on the base:

\begin{thm}
\label{thm:proper:Poisson:grpd}
Let $(\G,\pi_\G)\tto (M,\pi _M)$ be an orientable proper Poisson groupoid of type (C1) and assume that there exists a multiplicative 2-form extending its leafwise symplectic form. If $(\widetilde{\G},\widetilde{\pi_\G})$ is its universal covering groupoid, then there is a path of $\widetilde{\pi}_\G^{\,t}\in\X^2(\widetilde{\G})$ of Poisson structures starting at $\widetilde{\pi}_\G^{\,0}=\widetilde{\pi_\G}$ with the following properties:
\begin{enumerate}[(i)]
\item each $\widetilde{\pi}_\G^{\,t}$ is multiplicative of corank 1;
\item the Poisson structure on $M$ induced by $\widetilde{\pi}_\G^{\,t}$ is $\pi_M$;
\item $\widetilde{\pi}_\G^{\,1}$ is associated with a multiplicative cosymplectic structure.
\end{enumerate}
\end{thm}

The rest of this section will be dedicated to the proof of this theorem. 

Since $\pi_\G$ is a Poisson groupoid of type (C1), $\rho _{A^*}$ is surjective and there is a short exact sequence
\begin{equation}\label{eq:short:exact:sequence:A*}
\xymatrix{0\ar[r]& \mathfrak{k}^* \ar[r] & A^* \ar[r]^{\rho_{A^*}} & TM \ar[r] &0},
\end{equation}
and a dual sequence
\begin{equation}\label{eq:short:exact:sequence:A}
\xymatrix{0\ar[r]& T^*M \ar[r]^{\rho^*_{A^*}}  & A \ar[r] & \mathfrak{k} \ar[r] &0}.
\end{equation}

Let $\omega\in \Omega^2(\G)$ be the multiplicative 2-form extending the leafwise symplectic form, and $\al\in\Omega^1(\G)$ and $E\in\X(\G)$ be the corresponding multiplicative 1-form and bi-invariant vector field, given by Proposition \ref{prop:condition:a}. At the infinitesimal level these give:
\begin{enumerate}[(i)]
\item an IM 2-form $(\mu,\tilde{\mu}):A\to T^*M\oplus\wedge^2T^*M$;
\item an IM 1-form $(\nu,\tilde{\nu}):A\to \RR\oplus T^*M$;
\item a section $e\in{\Gamma}(\ker\rho_A)$ which is central: $[e,X]_A=0$, for any $X\in {\Gamma}(A)$.
\end{enumerate}

This data make the second short exact sequence \eqref{eq:short:exact:sequence:A} a split exact sequence of Lie algebroids. Indeed, one has $\mathfrak{k}=\RR e$, $\mu(e)=0$ and \eqref{eq:short:exact:sequence:A} becomes
\[
\xymatrix{0\ar[r]& M\times \RR \ar[r]^i & A \ar[r]^{\mu}\ar@/^1pc/[l]^{\nu} & T^*M \ar@/^1pc/[l]^{\rho^*_{A^*}}  \ar[r] &0}.
\]
where $i(x,\lambda)=\lambda e$. Because $e$ is central and $\rho^*_{A^*}:T^*M\to A$ is a Lie algebroid map, we conclude that we have a Lie algebroid isomorphism
\[ (\mu,\nu):A\diffto T^*M\oplus \RR, \]
where the right-hand side has anchor and bracket
\begin{align}
\rho _A(\alpha ,f)&=\pi_M^\sharp (\alpha ), \label{anchorA}\\
[(\alpha ,f),(\beta ,g)]_A&= ([\alpha ,\beta ]_{\pi_M}, 
\pi_M^\sharp (\alpha )(g) -\pi_M^\sharp (\beta )(f))., \label{bracketA}
\end{align}

Now let us look at the exact sequence \eqref{eq:short:exact:sequence:A*}. If $e^*$ is the section of $\mathfrak{k}^*$ defined by $\langle e^* , e\rangle =1$, then $\mathfrak{k}^*=\RR e^*$ and we have a vector bundle splitting
\[
\xymatrix{0  \ar[r] & M\times \RR \ar[r]^{\nu^*} & A^*  \ar[r]^{\rho_{A^*}}  \ar@/^1pc/[l]^{i^*}& TM \ar@/^1pc/[l]^{\mu^*}\ar[r] & 0} 
\]
The Lie algebroid structure of $A^*$ can be described in terms of the central section $e$.

\begin{lem}\label{lem:gamma:omega}
There is a Lie algebroid isomorphism
\[ (\rho_{A^*},i^*):A^*\diffto TM\oplus \RR, \]
where the right-hand side has anchor and Lie bracket given by
\begin{align}
\rho _{A^*}(X ,a)&=X, \label{anchorA*}\\
[ (X ,a),(Y ,b) ] _{A^*}&= ([X,Y ],\nabla _X b -\nabla _Y a +\Omega (X,Y) ), \label{bracketA*}
\end{align}
with $\nabla$ the flat connection on the trivial line bundle $M\times \RR\to M$ given by
\[
\nabla _X a= X(a)+a\gamma (X), \qquad \gamma (X):=\langle d_\ast e, X\wedge e^*\rangle, 
\]
and $\Omega \in \Omega ^2(M)$ given by
\[
\Omega (X,Y):=\langle d_\ast e, X\wedge Y\rangle.
\]
Moreover, for all $\alpha \in \Omega^1(M)$ one has
\[
i_{\pi_M^\sharp (\alpha )}\gamma =0,\quad i_{\pi_M^\sharp (\alpha )}\Omega=0.
\]
\end{lem}

\begin{rem}
Note that Jacobi identity for a Lie bracket of the form \eqref{bracketA*} is equivalent to the connection $\nabla$ being flat, i.e., to $\gamma$ being a closed 1-form, and the 2-form $\Omega$ being $\d^\nabla$-closed
\[ \d\gamma=0,\qquad \d^\nabla\Omega=0. \]
One can also expressed $\nabla$ and $\Omega$ in terms of the 2nd components of the IM forms associated with $(\omega,\al)$ as follows
\[ \gamma=\tilde{\nu}(e), \quad \Omega=\tilde{\mu}(e). \]
Hence, $A^*$ becomes the trivial extension of $TM$ precisely when $(\omega,\al)$ is cosymplectic, in agreement with Proposition \ref{prop:Lie:bialgebroid:cosymp}. 
\end{rem}

\begin{proof}[Proof of the Lemma  \ref{lem:gamma:omega}]
Under the identification $(\rho_{A^*},i^*):A^*\diffto TM\oplus \RR$, one has $\nabla _{X}e^*=[X,e^*]_{A^*}$ and $\Omega (X,Y)e^*=[X,Y]_{A^*}$. Hence, using the definition of $\d_*$, we find
\begin{align*}
\gamma (X)&= \langle e, [X,e^*]_{A^*}\rangle \\
&= \langle \d _*e, X\wedge e^* \rangle + X (\langle e,e^*\rangle )-
\rho _{A^*} (e^*) (\langle e,X\rangle )= \langle \d _* e,X\wedge e^* \rangle .\\
\Omega (X,Y)&= \langle e^*, [X,Y]_{A^*}\rangle\\ 
&= \langle \d _*e, X\wedge Y \rangle + X (\langle e,Y \rangle )-
Y (\langle e,X\rangle )= \langle \d _* e,X\wedge Y\rangle .
\end{align*}
On the other hand, by \cite[Cor. 3.9]{MX94}, we have
\[
{i_{\rho _A^*(\alpha )} \d_*e =[e, {\rho _{A^*}^*(\alpha )} ]_{A}  -  \d _* (\langle \alpha ,\rho _A (e)\rangle )- \rho ^*_{A^*} (i_{\rho _A (e)} \d\alpha ).}
\]
Observing that  $\rho^*_A (\alpha )=-\pi_M^\sharp (\alpha )$ and $\rho _A (e)=0$, the result follows.
\end{proof}

Next, we will see that $(A^*,A)$ is a triangular Lie bialgebroid in the sense of Mackenzie and Xu \cite{MX94}, i.e., there exists an element $\Lambda\in\Gamma(\wedge^2 A^*)$ satisfying 
\[ [\Lambda,\Lambda]_{A^*}=0,\] 
such that the anchor and Lie bracket on $A=(A^*)^*$ are given by
\begin{align}
\rho_A(\xi)&=\Lambda^\sharp(\xi),\label{anchorL}\\
[\xi_1,\xi_2]_A&=[\xi_1,\xi_2]_\Lambda:= \Lie _{\Lambda ^\sharp (\xi _1)}\xi _2-\Lie _{\Lambda ^\sharp (\xi _2)}\xi _1
-\d_{A} (\Lambda (\xi _1,\xi _2)). \label{bracketL}
\end{align}
In fact, we have the following general result which is an analogue for central extensions of the fact that for any Poisson structure $(M,\pi_M)$ the pair $(TM,T^*M)$ is a triangular Lie bialgebroid. 

\begin{prop}
\label{prop:triangular}
Let $\gamma\in\Omega^1(M)$, $\Omega\in\Omega^2(M)$ and $\pi_M\in\X^2(M)$, and denote by $\nabla$ the connection on the trivial line bundle given by $\gamma$. Assume that:
\begin{enumerate}[(i)]
\item $\gamma$ is closed: $\d\gamma=0$;
\item $\Omega$ is $\d^\nabla$-closed: $\d^\nabla\Omega=0$;
\item $\pi_M$ is Poisson: $[\pi _M,\pi _M]=0$.
\end{enumerate}
Then $A^*=TM\oplus\RR$, with anchor \eqref{anchorA*} and Lie bracket \eqref{bracketA*}, and  $A=T^*M\oplus\RR$ with anchor \eqref{anchorA} and Lie bracket \eqref{bracketA} are both Lie algebroids. If, additionally, one has 
\[
i_{\pi_M^\sharp (\alpha )}\gamma =0,\quad i_{\pi_M^\sharp (\alpha )}\Omega=0,\quad (\alpha \in \Omega^1(M)),
\]
then $(A^*,A)$ is a triangular Lie bialgebroid for the section $\Lambda \in \Gamma (\wedge^2 A^*)$ given by
\[
\Lambda ((\alpha ,f), (\beta ,g)):= \pi_M (\alpha ,\beta ).
\]
In particular, in this case one has
\[ [\Lambda,\Lambda]_{A^*}=0. \]
\end{prop}

\begin{proof}[Proof of Proposition \ref{prop:triangular}]
The fact that both $A$ and $A^*$ are Lie algebroids is standard. To check that under the additional assumptions on $\gamma$ and $\Omega$ the pair $(A^*,A)$ is a triangular bialgebroid, notice that
\[
\rho _A (\alpha ,f )=\pi _M^\sharp (\alpha)=\Lambda ^\sharp (\alpha , f), \qquad ((\alpha , f)\in \Gamma (A)),
\]
so \eqref{anchorL} holds. On the other hand, we find
\begin{align*}
\langle \Lie _{\Lambda ^\sharp (\alpha , f)} (\beta ,g ), (Y,b) \rangle &= \langle \Lie _{( \pi _M^\sharp (\alpha ),0)} (\beta ,g ), (Y,b) \rangle \\
&= \pi_M^\sharp (\alpha ) (\langle (\beta ,g), (Y,b)\rangle  ) - \langle (\beta ,g), [\pi_M^\sharp (\alpha ,0), (Y,b)]_A\rangle \\
&= \pi_M^\sharp (\alpha ) (\langle \beta , Y\rangle )+\pi_M^\sharp (\alpha ) (gb) +\\
&\qquad \qquad-\langle \beta , [\pi_M^\sharp (\alpha ) , Y] \rangle
-g \nabla _{\sharp (\alpha)}b-g\Omega (\pi_M^\sharp (\alpha),Y) \\
&= \langle \Lie _{\pi_M^\sharp (\alpha )}\beta -g\, i_{\pi_M^\sharp (\alpha )}\Omega , Y\rangle +
b \pi_M^\sharp (\alpha ) (g) - gb ( i_{\pi_M^\sharp (\alpha )}\gamma )\\
&= \langle \Lie _{\pi_M^\sharp (\alpha )}\beta  , Y\rangle +
b \pi_M^\sharp (\alpha ) (g) \rangle,
\end{align*}
where in the last line we have used the extra assumptions on on $\gamma$ and $\Omega$. Using this we find that the Lie bracket on $A$ is indeed given by  \eqref{bracketL}, namely
\begin{align*}
[(\alpha , f),(\beta ,g)]_\Lambda&=
\Lie _{\Lambda ^\sharp (\alpha , f)} (\beta ,g )-
\Lie _{\Lambda ^\sharp (\beta ,g)} (\alpha, f )-
\d _A (\Lambda ((\alpha , f) (\beta ,g ))) \\
&= (\Lie _{\pi_M^\sharp (\alpha )}\beta  , \pi_M^\sharp (\alpha )(g) ) - (\Lie _{\pi_M^\sharp (\beta )}\alpha  , \pi_M^\sharp (\beta )(f) )-(\d (\pi_M (\alpha ,\beta )),0) \\
&= ([\alpha ,\beta ]_{\pi _M} ,  \pi_M^\sharp (\alpha )(g)-\pi_M^\sharp (\beta )(f))\\
&= [(\alpha ,f),(\beta ,g)]_A.
\end{align*}

To complete the proof we show that $[\Lambda,\Lambda]_{A^*}=0$. For this we observe that by the computation above we have
\begin{align*}
[\Lambda,\Lambda]_{A^*}^\sharp((\alpha , f),(\beta ,g))&=\Lambda ^\sharp \big([(\alpha , f),(\beta ,g)]_\Lambda\big)-[\Lambda ^\sharp(\alpha , f),\Lambda ^\sharp(\beta ,g)]\\
&=\Lambda ^\sharp \big([(\alpha , f),(\beta ,g)]_A\big)-[\Lambda ^\sharp(\alpha , f),\Lambda ^\sharp(\beta ,g)]\\
&=\pi^\sharp_M([\alpha ,\beta ]_{\pi _M})-[\pi^\sharp_M(\alpha),\pi^\sharp_M(\beta)]\\
&=[\pi _M,\pi _M]^\sharp(\al,\be)=0,
\end{align*}
where the first identity is Lemma 2.2 in \cite{LX96}. 
\end{proof}

We can now complete the proof of Theorem \ref{thm:proper:Poisson:grpd}. We perform two consecutive homotopies of Lie bialgebroids as follows:
\begin{enumerate}
\item Starting with the original Poisson groupoid, its Lie bialgebroid is a triangular Lie bialgebroid $(A^*,A)$ as in Proposition \ref{prop:triangular} with associated data $(\gamma,\Omega,\pi _M)$. We can rescale the 2-form $\Omega$, obtaining a family of triangular Lie bialgebroids $(A^*_t,A)$ with data $(\gamma,(1-t)\Omega,\pi _M)$, $t\in [0,1]$ (note that this triple still satisfies for each $t$ all the conditions in the proposition). 
\item The previous homotopy gives at $t=1$ a Lie bialgebroid $(A^*_1,A)$with corresponding triple $(\gamma,\Omega=0,\pi_M)$. Now we can rescale the connection 1-form $\gamma$, obtaining a family of triangular Lie bialgebroids $(A^*_t,A)$ with data $((2-t)\gamma,0,\pi_M)$, $t\in [1,2]$ (notice again that this triple still satisfies for each $t$ all the conditions in the proposition).
\end{enumerate}
The result of these two consecutive deformations is a Lie bialgebroid $(A_2^*,A)$ whose associated triple has both $\gamma$ and $\Omega$ equal to zero, i..e, it is of cosymplectic type (cf.~Proposition \ref{prop:Lie:bialgebroid:cosymp}).

Finally, we observe that in these deformations $(A^*_t,A)$ the Lie bracket of $A$ and the anchors of both Lie algebroids are fixed, and so is the underling Poisson structure. Using the Mackenzie-Xu correspondence between Lie bialgebroids and source 1-connected Lie groupoids \cite{MX00}, we conclude that the Lie groupoid $\widetilde{\G}$ is fixed (since $A$ is fixed) and we have a path of multiplicative Poisson structures $\widetilde{\pi}_\G^{\,t}\in\X^2(\widetilde{\G})$ as in the statement of Theorem \ref{thm:proper:Poisson:grpd}.
\qed

\begin{rem}
A geometric way of thinking about the two deformations in the proof is as follows. We start with a Poisson groupoid $(\G,\pi_\G)$ which can be described by a pair $(\omega,\al)$ consisting of a multiplicative 2-form and a multiplicative 1-form, which fail to be closed but, nonetheless, $\ker\al$ is an integrable distribuition and the restriction of $\omega$ to the leaves of $\ker\al$ is symplectic. After replacing $\G$ by $\widetilde{\G}$, we are able to construct homotopies as follows:
\begin{enumerate}
\item the first homotopy consists of a deformation $(\omega_t,\al)$ where the 1-form $\al$ is fixed, the 2-form $\omega_t$ is multiplicative, at $t=0$ equals $\omega$ and at $t=1$ is closed;
\item the second homotopy consists of a deformation $(\omega_1,\al_t)$ where the 2-form $\omega_1$ is fixed, the 1-form $\al_t$ is multiplicative, at $t=0$ equals $\al$ and at $t=1$ is closed;
 \end{enumerate}
Moreover, through out these deformations the 1-form always defines an integral distribution and the restriction of the 2-form to its leaves is symplectic, so they define a multiplicative Poisson structure $\pi_t$ on $\G$.
\end{rem}

}

\subsection{Proper over-symplectic groupoids of corank 1}               %
\label{sec:oversymplectic}

{ 
Consider an oversymplectic groupoid $(\G,\omega)$ of corank 1 for which $\ker\omega$ is a simple foliation. Then we obtain an extension 
\[ \xymatrix{1\ar[r]& \K \ar[r] & (\G,\omega)\ar[r]^q & (\Sigma,\Omega_\Sigma) \ar[r] &1} \]
where $(\Sigma,\Omega_\Sigma)$ is a symplectic groupoid and $\omega=q^*\Omega_\Sigma$. If $\G$ is proper and orientable, then $\K$ is the trivial $\Ss^1$-bundle of groups and we have a $\Ss^1$-central extension 
\[ \xymatrix{1\ar[r]& M\times\Ss^1\ar[r] & \G\ar[r]^q & \Sigma \ar[r] &1}. \]
Notice that this also makes $\G$ into a $\Ss^1$-principal bundle and we denote the generator of the $\Ss^1$-action by $\partial_\theta\in\X(\G)$. A \emph{multiplicative Ehresmann connection} for such an extension is given by a multiplicative 1-form $\al\in\Omega^1(\G)$ with the property that:
\begin{equation} 
\label{eq:basic:connection}
i_{\partial_\theta}\al=1.
\end{equation}
We refer the reader to \cite{FM22,LSX09} for the theory of such connections and its relation to ordinary principal bundle connections. 

It is proved in \cite{FM22} that a $\Ss^1$-central extension of a proper groupoid always admits a multiplicative  Ehresmann connection $\al$. Its curvature 2-form is the multiplicative 2-form
\[ \Omega:=\d\al\in\Omega^2(\G). \]
This form is closed and so by \eqref{eq:basic:connection} it is basic. Hence, we have a multiplicative, closed, 2-form
$\underline{\Omega}\in\Omega^2(\Sigma)$ such that:
\[ \Omega=q^*\underline{\Omega}. \] 
Denoting by $H^\bullet_M(\Sigma)$ the \emph{multiplicative de Rham cohomology} of $\Sigma\tto M$, we have:

\begin{prop}
Given a $\Ss^1$-central extension of a proper Lie groupoid $\Sigma$
\[ \xymatrix{1\ar[r]& M\times\Ss^1\ar[r] & \G\ar[r]^q & \Sigma \ar[r] &1}, \]
the class of the basic curvature of a multiplicative  Ehresmann connection
\[ [\underline{\Omega}]\in H^2_M(\Sigma) \]
is independent of the choice of connection.
\end{prop}

\begin{proof}
If $\al_1$ and $\al_2$ are two multiplicative Ehresmann connections then their difference $\al_1-\al_2$ is a basic multiplicative 1-form, i.e., we have
\[ \al_1-\al_2=q^*\beta,\quad\text{with $\beta\in\Omega^1(\Sigma)$ multiplicative}. \]
It follows that their basic curvature 2-forms differ by an exact multiplicative form:
\[ \underline{\Omega}_1-\underline{\Omega}_2=q^*\d\beta. \]
\end{proof}

We call the class  $ [\underline{\Omega}]\in H^2_M(\Sigma)$ the \emph{multiplicative Chern class} of the extension. This class vanishes if and only if the extension admits a flat multiplicative Ehresmann connection.

\begin{thm}
\label{thm:oversymplectic:grpd}
Let $(\G,\omega)$ be a corank 1, orientable, proper oversymplectic groupoid. If $\ker\omega$ is a simple foliation, then there exists $\al\in\Omega^1(\G)$ such that $(\G,\omega,\al)$ is a cosymplectic groupoid if and only if the corresponding $\Ss^1$-central extension has vanishing multiplicative Chern class.
\end{thm}

\begin{proof}
If we can complete $\omega$ to a multiplcative cosymplectic stucture $(\omega,\al)$ then obviously $\al$ is a flat multiplicative Ehresmann connection.

For the reverse direction, assume that the multiplicative Chern class vanishes so the extension admits a multiplicative Ehresmann connection $\al$. Let $\dim \G=2n+1$, where $2n=\dim\Sigma$. From \eqref{eq:basic:connection} and the fact that $\omega=q^*\Omega_\Sigma$, with $\Omega_\Sigma$ non-degenerate, it follows that $\al\wedge\omega^n$ is nowhere vanishing. Hence, $(\omega,\al)$ is a multiplicative cosymplectic structure.
\end{proof}
}



\end{document}